\documentclass[11pt,reqno]{amsart}

\usepackage[T1]{fontenc}
\usepackage{lmodern}
\usepackage{amsmath,amssymb,mathrsfs}
\usepackage[a4paper,margin=29mm,headheight=14pt,headsep=18pt]{geometry}
\usepackage{microtype}
\usepackage[hidelinks]{hyperref}
\hypersetup{
  pdftitle={Resolutions of Compact Balanced Spaces},
  pdfsubject={Balanced metrics on resolutions and exceptional Aeppli classes},
  pdfkeywords={balanced space, resolution, Aeppli cohomology, negativity lemma}
}

\newtheorem{theorem}{Theorem}[section]
\newtheorem{proposition}[theorem]{Proposition}
\newtheorem{lemma}[theorem]{Lemma}
\newtheorem{corollary}[theorem]{Corollary}
\theoremstyle{definition}
\newtheorem{definition}[theorem]{Definition}
\theoremstyle{remark}
\newtheorem{remark}[theorem]{Remark}
\numberwithin{equation}{section}

\newcommand{\ii}{\sqrt{-1}}
\newcommand{\dd}{\mathrm d}
\newcommand{\dbar}{\bar\partial}
\newcommand{\R}{\mathbb R}
\newcommand{\C}{\mathbb C}
\newcommand{\Bclass}{\mathscr B}
\newcommand{\A}{\mathcal A}
\newcommand{\D}{\mathcal D}
\newcommand{\OO}{\mathcal O}
\DeclareMathOperator{\Supp}{Supp}
\DeclareMathOperator{\Span}{Span}

\DeclareMathOperator{\reg}{reg}
\DeclareMathOperator{\sing}{sing}

\title[Resolutions of compact balanced spaces]
{Resolutions of compact balanced spaces}

\author{Jixiang Fu}
\address{Institute of Mathematics, Fudan University, Shanghai 200433, China}
\email{majxfu@fudan.edu.cn}

\author{Wei Xia}
\address{Department of Mathematics, School of Artificial Intelligence, Taizhou University, Zhejiang 318000, China}
\email{yinfang@tzc.edu.cn}

\date{}

\begin{document}

\begin{abstract}
Let $X$ be a compact reduced balanced complex space of pure dimension $n\geq2$. We prove that every resolution of singularities of $X$ is a balanced manifold; in particular, $X$ belongs to the class $\Bclass$. This removes the additional cohomological assumptions appearing in earlier results of Fu--Meng--Xia. The main new ingredient is the linear independence of the Aeppli classes of the exceptional prime divisors of a projective bimeromorphic morphism onto a normal complex space. This follows from the vanishing of curve degrees for Aeppli-trivial divisors and the relative negativity lemma. For non-normal spaces, strict positivity is used to eliminate divisors mapping onto codimension-one subsets before passing to the normalization. We also recall a classical surface example of Grauert showing that the inclusion of compact balanced spaces in $\Bclass$ is strict.
\end{abstract}

\maketitle
\thispagestyle{plain}

\section{Introduction}

A Hermitian metric on a complex manifold of dimension $n\geq2$ is balanced if its associated real $(1,1)$-form $\omega$ satisfies
\[
\dd\omega^{n-1}=0.
\]
Michelsohn proved that the existence of such a metric is equivalent to the existence of a smooth, $\dd$-closed, strictly positive $(n-1,n-1)$-form~\cite{Michelsohn}. Fu--Meng--Xia extended this latter condition to reduced complex spaces and called the resulting spaces \emph{balanced spaces}. They also introduced the class $\Bclass$ of compact reduced complex spaces admitting a balanced resolution~\cite{FMX}.

A natural question is therefore how balancedness of a singular complex space is related to the existence of a balanced resolution. 
\begin{theorem}\label{thm:main}
Let $X$ be a compact reduced complex space of pure dimension $n\geq2$. Suppose that $X$ admits a smooth, $\dd$-closed, strictly positive $(n-1,n-1)$-form. Then every resolution of singularities of $X$ is a balanced manifold. In particular, $X\in\Bclass$.
\end{theorem}

In~\cite[Theorem~4.5]{FMX}, the conclusion $X\in\Bclass$ was obtained under the additional assumption that, for some resolution $M$, the natural map
\[
\iota_M:H^{1,1}_{BC}(M)\longrightarrow H^{1,1}_{A}(M)
\]
is injective. Theorem~\ref{thm:main} shows that no such global cohomological hypothesis is necessary. Only injectivity on the subspace spanned by the classes of exceptional divisors is needed.

The geometric reason for this restricted injectivity is the following. Let
\[
f:M\longrightarrow Z
\]
be a projective bimeromorphic morphism onto a normal complex space, and let $E_1,\ldots,E_r$ be its exceptional prime divisors. We prove that their Aeppli classes
\[
[E_1]_A,\ldots,[E_r]_A
\]
are linearly independent. If a real divisor
\[
D=\sum_j a_jE_j
\]
has vanishing Aeppli class, then its degree on every compact irreducible curve vanishes. In particular,
\[
D\equiv_f0.
\]
The relative negativity lemma then forces $D=0$. Thus the Bott--Chern-to-Aeppli map is automatically injective on the subspace generated by the exceptional divisor classes. 

The main additional issue for a non-normal space is that a divisor lying over the singular locus may map onto a codimension-one subset of $X$ and hence need not be exceptional. Pairing a positive current supported over the singular locus with the pullback of the balanced form shows that every component whose image has codimension one must have zero coefficient. The remaining components have images of codimension at least two; after passing to the normalization, they are exceptional divisors, so the preceding Aeppli-independence argument applies.

Theorem~\ref{thm:main} establishes the inclusion
\[
\{\text{compact balanced spaces}\}\subset\Bclass.
\]
This inclusion is strict already in complex dimension two. A classical surface contraction of Grauert~\cite{Grauert} provides such an example. Starting from a smooth projective ruled surface $\widetilde X$, one contracts a negative section to obtain a normal compact complex surface
\[
\pi:\widetilde X\longrightarrow Y
\]
which is not a K\"ahler analytic space. Since $\widetilde X$ is projective, it is K\"ahler and hence balanced. Thus $\pi$ is a balanced resolution and $Y\in\Bclass$. On the other hand, in complex dimension two every compact balanced complex space is K\"ahler. Thus $Y$ is not balanced, and therefore
\[
\{\text{compact balanced spaces}\}\subsetneq\Bclass.
\]

The paper is organized as follows. Section~\ref{sec:preliminaries} recalls smooth forms and positive currents on complex spaces, Michelsohn's characterization of balanced manifolds, and the support theorem for positive $\partial\dbar$-closed currents. Section~\ref{sec:exceptional} relates Aeppli triviality of divisor classes to the vanishing of their degrees on compact curves. Section~\ref{sec:negativity} recalls the relative negativity lemma. In Section~\ref{sec:exceptional-aeppli}, we use these results to establish the linear independence of exceptional Aeppli classes and the corresponding restricted injectivity of the Bott--Chern-to-Aeppli map. Section~\ref{sec:positivity} gives the positivity argument that detects and eliminates divisors whose images have codimension one. In Section~\ref{sec:resolution}, these ingredients are combined to prove Theorem~\ref{thm:main}, first in the irreducible case by passing to the normalization and then for general reduced spaces. Finally, Section~\ref{sec:strict-inclusion} recalls Grauert's surface construction, including the barycenter argument for non-projectivity, and uses it to show that the inclusion of compact balanced spaces in $\Bclass$ is strict.

\section{Forms and positive currents}\label{sec:preliminaries}
Let $X$ be a reduced complex space, with regular and singular loci $X_{\reg}$ and $X_{\sing}$. Given a local embedding $j:U\hookrightarrow V\subset\C^N$, smooth forms on $U$ are represented by smooth forms on $V$, where two representatives are identified if they agree after restriction to $U_{\reg}$. These local spaces define the sheaves of smooth forms on $X$. The operators $\partial$, $\dbar$, and $\dd$ are well defined, as is pullback by holomorphic maps, and pullback commutes with these operators. We use the conventions of~\cite{King}.

Let $M$ be a complex manifold. A real $(p,p)$-form on a complex manifold is \emph{weakly positive} if its restriction to every complex $p$-plane is a nonnegative top-degree form, and \emph{strictly positive} if that restriction is positive on every nonzero complex $p$-plane. A real smooth $(p,p)$-form $\Omega$ on $X$ is strictly positive if, near every point of $X$, there is a local embedding $j:U\hookrightarrow V$ and a smooth strictly positive $(p,p)$-form $\Omega_V$ on $V$ representing $\Omega|_U$.

\begin{definition}[{\cite[Definition~4.1]{FMX}}]
A reduced complex space $X$ of pure dimension $n\geq2$ is \emph{balanced} if it admits a smooth, $\dd$-closed, strictly positive $(n-1,n-1)$-form.
\end{definition}

A real $(1,1)$-current is the $(1,1)$-component of a boundary precisely when it can be written as
\[
T=\partial U+\dbar V, \qquad U\in\D'^{0,1}(M),\quad V\in\D'^{1,0}(M).
\]

\begin{theorem}[Michelsohn's criterion {\cite[Theorem~4.7]{Michelsohn}}]\label{thm:michelsohn}
A compact complex manifold $M$ of dimension $n\geq2$ is balanced if and only if there is no nonzero positive $(1,1)$-current on $M$ which is the $(1,1)$-component of a boundary.
\end{theorem}

The following form of the support theorem is the case needed here.
\begin{theorem}[Alessandrini--Bassanelli {\cite[Theorem~1.5]{AB-currents}}]\label{thm:support}
Let $M$ be a complex manifold of dimension $n$, and let $A\subset M$ be a compact analytic subset of dimension at most $n-1$. If $T$ is a positive $\partial\dbar$-closed $(1,1)$-current supported on $A$, then
\[
T=\sum_{j=1}^r a_j[E_j],\qquad a_j\geq0,
\]
where $E_1,\ldots,E_r$ are the $(n-1)$-dimensional irreducible components of $A$. In particular, $T=0$ if $\dim A\leq n-2$.
\end{theorem}

The following localization lemma is also used in the proof of~\cite[Theorem~4.5]{FMX}.
\begin{lemma}\label{lem:localization}
Let $M$ be a compact complex manifold of dimension $n$, and let $\eta$ be a smooth weakly positive $(n-1,n-1)$-form. Suppose that $\eta$ is strictly positive on $M\setminus A$, where $A$ is a closed subset. If $T$ is a positive $(1,1)$-current and $\langle T,\eta\rangle=0$, then $\Supp T\subset A$.
\end{lemma}
\begin{proof}
Let $U\Subset M\setminus A$ be an arbitrary relatively compact open set. Since $\eta$ is strictly positive on a neighborhood of $\overline U$, it dominates there a positive multiple of a fixed strictly positive $(n-1,n-1)$-form. Because $T$ is positive and $\eta$ is weakly positive on all of $M$, the equality $\langle T,\eta\rangle=0$ implies that $T$ has no mass on $U$. Since $U$ is arbitrary, $T$ vanishes on $M\setminus A$, and hence $\Supp T\subset A$.
\end{proof}

\section{Exceptional divisors and Aeppli cohomology}
\label{sec:exceptional}

For a compact complex manifold $M$, set
\begin{align*}
H^{1,1}_{BC}(M)
&=\frac{\{\theta\in\A^{1,1}(M):\dd\theta=0\}}{\partial\dbar\A^{0,0}(M)},\\
H^{1,1}_A(M)
&=\frac{\{\theta\in\A^{1,1}(M):\partial\dbar\theta=0\}}{\partial\A^{0,1}(M)+\dbar\A^{1,0}(M)}.
\end{align*}
The identity induces a natural map
\[
\iota_M:H^{1,1}_{BC}(M)\longrightarrow H^{1,1}_A(M).
\]

The inclusion of smooth forms into currents induces natural isomorphisms in both Bott--Chern and Aeppli cohomology. We use these identifications without further comment.

Let
\[
D=\sum_j a_jD_j
\]
be a real divisor on $M$, with distinct prime divisors $D_j$ and coefficients $a_j\in\R$. We denote by
\[
[D]=\sum_j a_j[D_j]
\]
the associated integration current. Since $[D]$ is a $\dd$-closed real $(1,1)$-current, it determines classes
\[
[D]_{BC}\in H^{1,1}_{BC}(M), \qquad [D]_A\in H^{1,1}_A(M).
\]

For a compact irreducible curve $C\subset M$, with normalization map $\gamma:\widehat C\to M$, define
\[
D\cdot C:=\sum_j a_j\deg\gamma^*\OO_M(D_j).
\]
This definition also applies when $C$ is contained in the support of $D$.

\begin{lemma}\label{lem:aeppli-degrees}
Let $D$ be a real divisor on a compact complex manifold $M$. If
\[
[D]_A=0,
\]
then
\[
D\cdot C=0
\]
for every compact irreducible curve $C\subset M$.
\end{lemma}
\begin{proof}
Write $D=\sum_j a_jD_j$. Choose a smooth Hermitian metric $h_j$ on $\OO_M(D_j)$ and let $s_j$ be its canonical section. Set
\[
\theta_j=c_1(\OO_M(D_j),h_j),\qquad \theta_D=\sum_j a_j\theta_j.
\]
With the usual normalization of the first Chern form, the Poincar\'e--Lelong formula gives
\[
[D]-\theta_D=\frac{\ii}{2\pi}\partial\dbar\left(\sum_j a_j\log|s_j|_{h_j}^2\right).
\]
Consequently, $[D]$ and the smooth Chern form $\theta_D$ determine the same Aeppli class. Hence $[D]_A=0$ implies $[\theta_D]_A=0$. Thus there exist smooth forms $\alpha\in\A^{0,1}(M)$ and $\beta\in\A^{1,0}(M)$ such that
\[
\theta_D=\partial\alpha+\dbar\beta.
\]
Let $\gamma:\widehat C\to M$ be the normalization of a compact irreducible curve. Then
\[
\gamma^*\theta_D=\partial(\gamma^*\alpha)+\dbar(\gamma^*\beta)=\dd\bigl(\gamma^*\alpha+\gamma^*\beta\bigr).
\]
The last equality holds because the omitted terms have bidegrees $(0,2)$ and $(2,0)$ on the Riemann surface $\widehat C$. Chern--Weil theory and Stokes' theorem now give
\[
D\cdot C=\int_{\widehat C}\gamma^*\theta_D=0.
\]
\end{proof}

\section{Relative divisors and the negativity lemma}
\label{sec:negativity}

\begin{definition}
Let $f:Y\to Z$ be a proper holomorphic map and let $D$ be a real Cartier divisor on $Y$.
\begin{enumerate}
\item A compact irreducible curve $C\subset Y$ is \emph{contracted by $f$} if $f(C)$ is a point;
\item the divisor $D$ is \emph{$f$-nef} if
\[
D\cdot C\geq0
\]
for every compact irreducible curve contracted by $f$;
\item we write
\[
D\equiv_f0
\]
if
\[
D\cdot C=0
\]
for every compact irreducible curve contracted by $f$;
\item a real divisor on $Y$ is \emph{$f$-exceptional} if every prime divisor $E$ appearing in its support satisfies
\[
\operatorname{codim}_Z f(E)\geq2.
\]
\end{enumerate}
\end{definition}

\begin{theorem}[Analytic negativity lemma]\label{thm:negativity}
Let $f:Y\to Z$ be a projective bimeromorphic morphism of normal complex spaces, and let $D$ be an $f$-exceptional real Cartier divisor. If $D$ is $f$-nef, then $D\leq0$. Consequently, $D\equiv_f0$ implies $D=0$.
\end{theorem}
\begin{proof}
In the algebraic setting, this is BCHM, Lemma~3.6.2 (``Negativity of contraction'')~\cite[Lemma~3.6.2]{BCHM}. Fujino explains that BCHM, Lemma~3.6.2, remains valid for projective morphisms of normal complex analytic spaces; see~\cite[Section~11, immediately after Definition~11.1]{Fujino}. Since the statement used here is local over $Z$, one may shrink to suitable Stein neighborhoods and apply that analytic version.

To deduce the stated form from BCHM, Lemma~3.6.2(2), write
\[
D=D^+-D^-,
\]
where $D^+\geq0$ and $D^-\geq0$ have no common irreducible components. Because $D$ is $f$-exceptional, both $D^+$ and $D^-$ are $f$-exceptional.

BCHM, Lemma~3.6.2(2), in the analytic setting just cited, gives the following statement. Suppose
\[
f^*L\equiv_f M+G+E,
\]
where $L$ is a real Cartier divisor on $Z$, $M$ is an $f$-nef real Cartier divisor on $Y$, $G\geq0$, $E$ is $f$-exceptional, and $G$ and $E$ have no common components. Then
\[
E\geq0.
\]

Apply this with
\[
L=0,\qquad M=D,\qquad G=D^-,\qquad E=-D^+.
\]
Indeed,
\[
f^*L=0\equiv_f D+D^- -D^+=M+G+E.
\]
The divisor $M=D$ is $f$-nef by assumption, $G=D^-\geq0$, $E=-D^+$ is $f$-exceptional, and $G$ and $E$ have no common components. Hence the negativity lemma gives
\[
-D^+=E\geq0.
\]
Since also $D^+\geq0$, it follows that
\[
D^+=0.
\]
Therefore,
\[
D=-D^-\leq0.
\]

Finally, assume
\[
D\equiv_f0.
\]
Then both $D$ and $-D$ are $f$-nef and $f$-exceptional real Cartier divisors. Applying the first part to $D$ and to $-D$ gives
\[
D\leq0, \qquad -D\leq0.
\]
Consequently,
\[
D=0.
\]
\end{proof}

\section{Exceptional Aeppli classes}
\label{sec:exceptional-aeppli}

\begin{proposition}\label{prop:exceptional-independence}
Let $f:M\to Z$ be a projective bimeromorphic morphism from a compact complex manifold of dimension $n$ onto a normal compact complex space. Let $E_1,\ldots,E_r$ be distinct $f$-exceptional prime divisors. Then
\[
[E_1]_A,\ldots,[E_r]_A
\]
are linearly independent over $\R$.
\end{proposition}
\begin{proof}
Suppose
\[
\sum_{j=1}^r a_j[E_j]_A=0, \qquad a_j\in\R,
\]
and set
\[
D:=\sum_{j=1}^r a_jE_j.
\]
The assumed relation is equivalent to
\[
[D]_A=0.
\]
By Lemma~\ref{lem:aeppli-degrees},
\[
D\cdot C=0
\]
for every compact irreducible curve $C\subset M$. In particular, this holds for every curve contracted by $f$, so
\[
D\equiv_f0.
\]
The divisor $D$ is $f$-exceptional because every component of its support is one of the $E_j$. Theorem~\ref{thm:negativity} therefore gives
\[
D=0.
\]
Since the $E_j$ are distinct prime divisors, all coefficients $a_j$ are zero.
\end{proof}

\begin{corollary}
Under the assumptions of Proposition~\ref{prop:exceptional-independence}, the natural map
\[
\iota_M:H^{1,1}_{BC}(M,\R)\longrightarrow H^{1,1}_{A}(M,\R)
\]
is injective on
\[
\Span_{\R}\{[E_1]_{BC},\ldots,[E_r]_{BC}\}.
\]
In particular, no nonzero nonnegative linear combination of the currents $[E_j]$ is the $(1,1)$-component of a boundary.
\end{corollary}

\begin{proof}
A class in the indicated span that maps to zero in Aeppli cohomology produces a relation
\[
\sum_j a_j[E_j]_A=0.
\]
Proposition~\ref{prop:exceptional-independence} forces every $a_j$ to vanish. The final assertion follows because being the $(1,1)$-component of a boundary is exactly Aeppli-exactness for a real $(1,1)$-current.
\end{proof}

\section{Strict positivity detects nonexceptional divisors}
\label{sec:positivity}

For a non-normal target, divisors lying over the singular locus need not be exceptional. The following lemma deals with the components whose images have codimension one.

\begin{lemma}\label{lem:mass}
Let $X$ be a compact reduced complex space of pure dimension $n$, let $\Omega$ be a smooth strictly positive $(n-1,n-1)$-form on $X$, and let $\pi:M\to X$ be a proper holomorphic map from a compact complex $n$-manifold. Then for every prime divisor $E\subset M$,
\[
\int_E\pi^*\Omega\geq0.
\]
More precisely,
\[
\int_E \pi^*\Omega>0\quad\text{if }\dim\pi(E)=n-1,
\]
whereas
\[
\int_E \pi^*\Omega=0\quad\text{if }\dim\pi(E)\leq n-2.
\]
\end{lemma}
\begin{proof}
The pullback $\pi^*\Omega$ is weakly positive, so its integral over $E$ is nonnegative. Set $V:=\pi(E)$ with the reduced structure and choose a resolution $\rho:\widehat E\to E$. Let
\[
g:\widehat E\longrightarrow V\hookrightarrow X
\]
be the composite map. Then
\[
\int_E\pi^*\Omega=\int_{\widehat E}g^*\Omega.
\]
If $\dim V=n-1$, then $g$ is generically finite onto $V$. On a nonempty open subset of $\widehat E$, its differential has rank $n-1$. A local ambient strictly positive representative of $\Omega$ therefore pulls back there to a strictly positive volume form. Elsewhere the pullback is nonnegative, so the integral is strictly positive.

If $\dim V\leq n-2$, then the generic rank of $g$ is at most $n-2$. The pullback of an $(n-1,n-1)$-form therefore vanishes on a dense open subset of $\widehat E$. Since $g^*\Omega$ is smooth, it vanishes everywhere, and the integral is zero.
\end{proof}

\section{Proof of the resolution theorem}\label{sec:resolution}

\begin{proof}[Proof of Theorem~\ref{thm:main}]
Let $\Omega$ be a balanced form on $X$. Assume first that $X$ is irreducible. Let
\[
\nu:X^\nu\longrightarrow X
\]
be the normalization. This is a finite bimeromorphic morphism, and $X^\nu$ is a normal compact complex space. Resolution of singularities in the analytic category provides a resolution
\[
M\xrightarrow{\mu}X^\nu\xrightarrow{\nu}X, \qquad \pi=\nu\circ\mu,
\]
where $\mu$ is projective and is an isomorphism over $(X^\nu)_{\reg}$. One may choose $\mu$ to be a finite composition of blowups; compactness makes the sequence finite. We use the analytic resolution theorem in this form; see~\cite{BierstoneMilman}. Since $\nu$ is an isomorphism over $X_{\reg}$, so is $\pi$.

By Theorem~\ref{thm:michelsohn}, to prove that $M$ is balanced, it is enough to show that every positive real $(1,1)$-current that is the $(1,1)$-component of a boundary vanishes. Since the current is real, we may write
\[
T\geq0, \qquad T=\partial\overline S+\dbar S, \qquad S\in\D'^{1,0}(M).
\]
Because $\pi^*\Omega$ is $\dd$-closed,
\[
\langle T,\pi^*\Omega\rangle=0.
\]
Moreover, $\pi^*\Omega$ is weakly positive everywhere and strictly positive on $\pi^{-1}(X_{\reg})$. Hence Lemma~\ref{lem:localization} gives
\[
\Supp T\subset \pi^{-1}(X_{\sing}).
\]
Since
\[
\partial\dbar T=0,
\]
Theorem~\ref{thm:support} gives
\begin{equation}\label{eq:T-divisor}
T=\sum_{j=1}^r a_j[E_j], \qquad a_j\geq0,
\end{equation}
where $E_1,\ldots,E_r$ are the divisorial irreducible components of $\pi^{-1}(X_{\sing})$. Set
\[
D:=\sum_{j=1}^r a_jE_j.
\]
Then $T=[D]$. Since $T$ is the $(1,1)$-component of a boundary,
\[
[D]_A=0.
\]

If $X$ is normal, then every component occurring in $D$ satisfies
\[
\dim\pi(E_j)\leq n-2.
\]
Thus each $E_j$ is $\pi$-exceptional, and Proposition~\ref{prop:exceptional-independence} yields $D=0$.

If $X$ is non-normal, $X_{\sing}$ may have codimension-one components. Pairing~\eqref{eq:T-divisor} with $\pi^*\Omega$ yields
\[
0=\sum_{j=1}^r a_j\int_{E_j}\pi^*\Omega.
\]
Every summand is nonnegative. By Lemma~\ref{lem:mass}, if $\dim\pi(E_j)=n-1$, then its integral is strictly positive, and hence $a_j=0$. Therefore every component occurring in $D$ satisfies
\[
\dim\pi(E_j)\leq n-2.
\]
The normalization map $\nu$ is finite and preserves the dimension of analytic subspaces. Consequently,
\[
\dim\mu(E_j)=\dim\nu(\mu(E_j))=\dim\pi(E_j)\leq n-2
\]
for every component of $D$. Thus $D$ is $\mu$-exceptional. Applying Proposition~\ref{prop:exceptional-independence} to $\mu:M\to X^\nu$, we obtain $D=0$.

For a general reduced $X$ of pure dimension $n$, let $X_1,\ldots,X_s$ be its irreducible components, each equipped with its reduced structure. The restriction of $\Omega$ is a balanced form on every $X_i$. By the preceding argument, each $X_i$ has a balanced resolution $M_i$. The disjoint union $\coprod_i M_i\to X$ is therefore a balanced resolution of $X$.

Finally, let $M'\to X$ be any resolution. Resolving the closure of the graph of the induced bimeromorphic map between $M'$ and the balanced resolution constructed above gives a common smooth modification. Balancedness is invariant under modifications of compact complex manifolds by~\cite{AB-modifications}; applying this result to the two projections shows that $M'$ is balanced. This proves the theorem.
\end{proof}

\section{Balanced spaces form a proper subclass of \texorpdfstring{$\Bclass$}{B}}
\label{sec:strict-inclusion}

Theorem~\ref{thm:main} proves that every compact reduced balanced space of pure dimension at least two belongs to $\Bclass$. The converse is false. A classical construction of Grauert~\cite[\S8(d), pp.~365--366]{Grauert} gives a counterexample.

\begin{proposition}
Among compact reduced complex spaces of pure dimension at least two, one has the strict inclusion
\[
\{\text{compact balanced spaces}\}\subsetneq\Bclass.
\]
\end{proposition}

\begin{proof}
Let $R$ be a compact Riemann surface of genus $2$, and fix a point $p\in R$. Set
\[
F:=\OO_R(p).
\]
Then $\deg F=1$. By the Riemann--Roch theorem,
\[
h^0(R,F)-h^1(R,F)=\deg F+1-g(R)=0.
\]
Since $F$ has a nonzero holomorphic section, $h^0(R,F)>0$, and hence
\[
H^1(R,F)\neq0.
\]
Choose a nonzero class
\[
0\neq\xi\in H^1(R,F).
\]

Let $\{U_i\}$ be a finite open covering of $R$ over which $F$ is trivial, let $g_{ij}$ be the transition functions of $F$, and represent $\xi$ by a \v{C}ech cocycle $\{\xi_{ij}\}$. Glue the products $U_i\times\C$ by the affine transformations
\[
z_i=g_{ij}z_j+\xi_{ij}.
\]
This gives a holomorphic affine-line bundle
\[
q:A\longrightarrow R
\]
whose structure group is the group of complex affine transformations. Equivalently, $A$ is an $F$-torsor whose class is $\xi$.

The choice $\xi\neq0$ means precisely that $A$ has no holomorphic section. Indeed, a holomorphic section would be represented by holomorphic functions $s_i$ on $U_i$ satisfying
\[
s_i=g_{ij}s_j+\xi_{ij}.
\]
Thus the cocycle $\{\xi_{ij}\}$ would be a coboundary, contrary to the choice of $\xi$.

We now compactify every affine fiber $\C$ by adding its point at infinity. Since affine transformations extend to automorphisms of $\mathbb P^1$ fixing infinity, this produces a holomorphic $\mathbb P^1$-bundle
\[
\widetilde X\longrightarrow R
\]
together with a distinguished section
\[
C:=\widetilde X\setminus A\simeq R,
\]
the section at infinity.

The same construction may be described in terms of the extension associated with $\xi$,
\[
0\longrightarrow F\longrightarrow E\longrightarrow\OO_R\longrightarrow0.
\]
If $\mathbb P(E)$ denotes the bundle of lines in the fibers of $E$, then
\[
\widetilde X\simeq\mathbb P(E),
\]
and $C$ is the section corresponding to the line subbundle $F\subset E$. In particular, $\widetilde X$ is a smooth projective ruled surface. Moreover,
\[
N_{C/\widetilde X}\simeq F^{-1},
\]
and therefore
\[
C^2=\deg F^{-1}=-1.
\]

By Grauert's contraction criterion~\cite{Grauert}, the curve $C$ can be contracted to a point. Thus there is a proper bimeromorphic holomorphic map
\[
\pi:\widetilde X\longrightarrow Y
\]
onto a normal compact complex surface $Y$ such that
\[
\pi(C)=\{y_0\}
\]
and
\[
\pi:\widetilde X\setminus C
   \longrightarrow
   Y\setminus\{y_0\}
\]
is biholomorphic. In particular, $\pi$ is a resolution of the unique singular point $y_0$ of $Y$.

Following Grauert~\cite[\S8(d), pp.~365--366]{Grauert}, suppose for contradiction that $Y$ is projective. Choose a projective embedding of $Y$ and a general hyperplane section
\[
B\subset Y
\]
which does not contain $y_0$. Its inverse image
\[
\widetilde B:=\pi^{-1}(B)
\]
is a compact curve contained in
\[
\widetilde X\setminus C=A.
\]
The restriction
\[
q|_{\widetilde B}:\widetilde B\longrightarrow R
\]
is finite and surjective. Indeed, its image cannot be a point, since a positive-dimensional compact analytic subset contained in a fiber of
\[
A\longrightarrow R
\]
would be a compact positive-dimensional analytic subset of $\C$, which is impossible.

Let $m$ be the degree of $q|_{\widetilde B}$. Over a sufficiently small open subset $U_i\subset R$, the intersection of $\widetilde B$ with a fiber of $A$ consists, with multiplicities, of $m$ points
\[
z_{i,1}(x),\ldots,z_{i,m}(x).
\]
Their barycenter
\[
s_i(x):=\frac1m\sum_{\nu=1}^m z_{i,\nu}(x)
\]
depends holomorphically on $x$. On an overlap $U_i\cap U_j$, the affine transition rule gives
\[
z_{i,\nu}=g_{ij}z_{j,\nu}+\xi_{ij},
\]
and hence
\[
s_i=g_{ij}s_j+\xi_{ij}.
\]
Thus the collection $\{s_i\}$ defines a holomorphic section of the affine bundle
\[
A\longrightarrow R.
\]
This contradicts $\xi\neq0$. Therefore $Y$ is not projective.

Moishezon proved that $Y$ is not a K\"ahler analytic space; see~\cite{Moishezon}, and see also the discussion of Grauert's example in~\cite[Section~2, pp.~2989]{FilipTosatti}. Since $\dim_{\C}Y=2$, every balanced form on $Y$ is a smooth, $\dd$-closed, strictly positive $(1,1)$-form, hence a K\"ahler form in the sense of complex spaces. Thus $Y$ is not balanced.

On the other hand, $\widetilde X$ is projective, hence K\"ahler and therefore balanced. Since
\[
\pi:\widetilde X\longrightarrow Y
\]
is a resolution, it is a balanced resolution of $Y$. Hence
\[
Y\in\Bclass.
\]
Therefore
\[
Y\in
\Bclass\setminus
\{\text{compact balanced spaces}\},
\]
and consequently
\[
\{\text{compact balanced spaces}\}
\subsetneq
\Bclass.
\]
\end{proof}

\begin{remark}
A smooth non-K\"ahler Moishezon manifold is bimeromorphic to a projective manifold and is therefore balanced, by the invariance of balancedness under modifications~\cite{AB-modifications}. Such examples do not separate balanced manifolds from $\Bclass$.

The example above also shows that the modification theorem of Alessandrini--Bassanelli~\cite{AB-modifications} cannot be extended naively by allowing the target to be singular: $\widetilde X$ is projective, hence balanced, whereas its singular bimeromorphic image $Y$ is not balanced.
\end{remark}

\end{document}